\documentclass[12pt]{article}

\usepackage{amsthm, amsmath, amssymb}

\numberwithin{equation}{section}

\theoremstyle{plain}	     
\newtheorem{thm}{Theorem}[section] 

\newtheorem{lem}[thm]{Lemma}
\newtheorem{prop}[thm]{Proposition}
     
\theoremstyle{definition}

\theoremstyle{remark} 
\newtheorem{rem}[thm]{Remark}
	
\newcommand{\disp}{\displaystyle}

\begin{document}
\title{Complete $(p,q)$-elliptic integrals \\ with application to a family of means
\footnote{The work of S.\ Takeuchi 
was supported by JSPS KAKENHI Grant Number 24540218.}}
\author{Toshiki Kamiya and Shingo Takeuchi\footnote{Corresponding author} \\
Department of Mathematical Sciences\\
Shibaura Institute of Technology
\thanks{307 Fukasaku, Minuma-ku,
Saitama-shi, Saitama 337-8570, Japan. \endgraf
{\it E-mail address\/}: shingo@shibaura-it.ac.jp (S.\ Takeuchi) \endgraf
}}
\date{}

\maketitle

\begin{abstract}
The complete elliptic integrals are generalized by 
using the generalized trigonometric functions with two parameters.
It is shown that a particular relation holds for the generalized integrals.
Moreover, as an application of the integrals, 
an alternative proof of a result  
for a family of means by Bhatia and Li, 
which involves the logarithmic mean and the 
arithmetic-geometric mean, is given. 
\end{abstract}



\section{Introduction}

In this paper, 
we deal with a \textit{complete $(p,q)$-elliptic integral 
of the first kind}
$$\mathcal{K}_{p,q}(k):=\int_0^{\frac{\pi_{p,q}}{2}}
\frac{d\theta}{(1-k^q\sin_{p,q}^{q}{\theta})^{1-\frac1p}}
=\int_0^1 \frac{dt}{(1-t^q)^\frac1p (1-k^qt^q)^{1-\frac1p}},$$
where $\sin_{p,q}{\theta}$ is the generalized $(p,q)$-trigonometric function
and $\pi_{p,q}$ denotes the half-period of $\sin_{p,q}{\theta}$.
The function $\sin_{p,q}{\theta}$ and the number $\pi_{p,q}$ 
play important roles to express the solutions $(\lambda,u)$ 
of inhomogeneous eigenvalue problem of $p$-Laplacian
$-(|u'|^{p-2}u')'=\lambda |u|^{q-2}u$
with a boundary condition.
See Section 2 for the definition of $\sin_{p,q}{\theta}$ and $\pi_{p,q}$; 
also \cite{DM,LE,T,T2} for details.
For $p=q=2$, it is easy to see that $\sin_{p,q}{\theta},\ \pi_{p,q}$ and
$\mathcal{K}_{p,q}(k)$ are identical to the classical $\sin{\theta},\ \pi$ and 
$\mathcal{K}(k)$, respectively, where $\mathcal{K}(k)$ is the complete elliptic 
integral of the first kind
$$\mathcal{K}(k):=\int_0^{\frac{\pi}{2}}
\frac{d\theta}{\sqrt{1-k^2\sin^2{\theta}}}
=\int_0^1 \frac{dt}{\sqrt{(1-t^2)(1-k^2t^2)}}.$$
Moreover, $\mathcal{K}_{p,q}(k)$ for $p=q$ has been already studied in \cite{T3}.

%


In this paper 
we will apply the complete $(p,q)$-elliptic integral $\mathcal{K}_{p,q}(k)$ to study 
a family of means defined by Bhatia and Li \cite{BL}
and to give an alternative proof of their theorem.

For a while, we will describe a part of the study in \cite{BL}.
Let $a$ and $b$ be positive numbers. 
The \textit{logarithmic mean} $\mathrm{L}(a,b)$ of $a$ and $b$
is defined by
$$\mathrm{L}(a,b):=
\begin{cases}
\dfrac{a-b}{\log{a}-\log{b}} & (a \neq b),\\
a & (a=b).
\end{cases}
$$
The \textit{arithmetic-geometric mean} $\mathrm{AG}(a,b)$
of $a$ and $b$ is defined as follows:
Let us consider the sequences $\{a_n\}$ and $\{b_n\}$ satisfying
$$a_{n+1}=\frac{a_n+b_n}{2}, \quad b_{n+1}=\sqrt{a_nb_n}, \quad n=0,1,2,\ldots$$
with $a_0=a$ and $b_0=b$. The sequences 
$\{a_n\}$ and $\{b_n\}$ converge to a common limit, and
$$\mathrm{AG}(a,b):=\lim_{n \to \infty}a_n=\lim_{n \to \infty}b_n.$$

It is known that $\mathrm{L}(a,b)$ and $\mathrm{AG}(a,b)$ 
have integral expressions as
\begin{gather*}
\frac{1}{\mathrm{L}(a,b)}=\int_0^\infty \frac{dt}{(t+a)(t+b)},\\
\frac{1}{\mathrm{AG}(a,b)}=\frac{2}{\pi} \int_0^\infty
\frac{dt}{\sqrt{(t^2+a^2)(t^2+b^2)}}.
\end{gather*}
Indeed, the first one follows from direct calculation of the right-hand side 
and the second one is a celebrated result of Gauss 
(e.g., \cite[Theorem 3.2.3]{AAR} or \cite[Theorem 1.1]{BB3} with 
setting $b\tan{\theta}=t$). 

Motivated by these expressions, 
Bhatia and Li introduced an interpolating family of means $\mathrm{M}_p(a,b)$ by
\begin{equation*}
\frac{1}{\mathrm{M}_p(a,b)}:=c_p \int_0^\infty
\frac{dt}{((t^p+a^p)(t^p+b^p))^{\frac{1}{p}}},\quad p \in (0,\infty),
\end{equation*}
where $c_p$ is defined to satisfy $\mathrm{M}_p(a,a)=a$, hence,
$$\frac{1}{c_p}:=\int_0^\infty \frac{dt}{(1+t^p)^{\frac{2}{p}}}.$$
Moreover, $\mathrm{M}_0$ is defined by taking limit:
$$\mathrm{M}_0(a,b)=\lim_{p \to +0}\mathrm{M}_p(a,b)=\sqrt{ab}.$$
Clealy, $\mathrm{M}_1(a,b)=\mathrm{L}(a,b)$ and 
$\mathrm{M}_2(a,b)=\mathrm{AG}(a,b)$, thus
$\mathrm{M}_p(a,b)$ is a generalization of $\mathrm{L}(a,b)$ and $\mathrm{AG}(a,b)$.
It is easily seen that 
$\mathrm{M}_p(a,b)$ is a \textit{binary symmetric mean} of positive numbers $a$ and $b$,
that is 
\begin{enumerate}
\item $\min\{a,b\} \leq \mathrm{M}_p(a,b) \leq \max\{a,b\}$
\item $\mathrm{M}_p(a,b)=\mathrm{M}_p(b,a)$
\item $\mathrm{M}_p(\alpha a,\alpha b)=\alpha \mathrm{M}_p(a,b)$ for all $\alpha>0$
\item $\mathrm{M}_p(a,b)$ is non-decreasing in $a$ and $b$.
\end{enumerate} 

They studied relation between $\mathrm{M}_p(a,b)$ and $\mathrm{K}_p(a,b)$,
the \textit{power difference mean} of $a$ and $b$.
This is defined for any $p \in \mathbb{R}$ and $a,\ b>0$ by
$$\mathrm{K}_p(a,b):=
\begin{cases}
\dfrac{p-1}{p}\dfrac{a^p-b^p}{a^{p-1}-b^{p-1}} & (a \neq b),\\
a & (a=b),
\end{cases}
$$
where it is understood that
\begin{align*}
\mathrm{K}_0(a,b)&:=\lim_{p \to 0}\mathrm{K}_p(a,b)=\frac{ab}{\mathrm{L}(a,b)},\\
\mathrm{K}_1(a,b)&:=\lim_{p \to 1}\mathrm{K}_p(a,b)=\mathrm{L}(a,b).
\end{align*}
For more details of $\mathrm{K}_p(a,b)$, 
see \cite{BL,Na} and the references given there.

These two means are related in the following sense.
It is easy to check 
that $1/\mathrm{M}_p(a,b)$ can be written as $(t^p+a^p=a^ps^{-1})$
\begin{equation*}
\frac{1}{\mathrm{M}_p(a,b)}=
\frac{\disp \int_0^1\frac{s^{\frac1p-1}(1-s)^{\frac1p-1}}{(a^p(1-s)+b^ps)^\frac1p}\,ds}
{\disp \int_0^1s^{\frac1p-1}(1-s)^{\frac1p-1}\,ds}
\end{equation*}
and $1/\mathrm{K}_p(a,b)$ also admits the following integral expression:
\begin{equation}
\label{eq:BLkp}
\frac{1}{\mathrm{K}_p(a,b)}=\int_0^1 \frac{ds}{(a^p(1-s)+b^ps)^{\frac1p}}.
\end{equation}
Thus $1/\mathrm{M}_p(a,b)$ and $1/\mathrm{K}_p(a,b)$ are
the weighted mean with the beta distribution
and with the continuous uniform distribution of
the function $(a^p(1-s)+b^ps)^{-1/p}$, respectively.

They concluded the following theorem
with easy but technical calculation. 
But, it is hard to say that these calculations are natural.

\begin{thm}[\cite{BL}]
\label{thm:BL4}
Given $a,\ b>0$ and $a \neq b$, we have
\begin{enumerate}
\item $\mathrm{M}_p(a,b)>\mathrm{K}_p(a,b)$ if $0 \leq p<1$
\item $\mathrm{M}_1(a,b)=\mathrm{K}_1(a,b)$
\item $\mathrm{M}_p(a,b)<\mathrm{K}_p(a,b)$ if $p>1$. 
\end{enumerate}
\end{thm}


In this paper, we will give an alternative proof of Theorem \ref{thm:BL4}.
Using the complete $(p,q)$-elliptic integral, 
we can easily give a hypergeometric representation
\eqref{eq:hypergeometric} in Theorem \ref{thm:main} below for $1/\mathrm{M}_p(a,b)$. 
Applying a formula 
of hypergeometric function to \eqref{eq:hypergeometric} and \eqref{eq:hypergeometric2},
we have \eqref{eq:mp} and \eqref{eq:kp}.
We emphasize that Theorem \ref{thm:BL4} of Bhatia and Li 
follows immediately from Theorem \ref{thm:main} with comparing only 
the third parameters of \eqref{eq:mp} and \eqref{eq:kp}.
\begin{thm}
\label{thm:main}
Let $p \in (0,\infty),\ p \neq 1$ and $x \in (0,1]$. Then
\begin{align}
\frac{1}{\mathrm{M}_p(1,x)}
\notag
&=\frac{2}{\pi_{p^*,p}}\mathcal{K}_{p^*,p}((1-x^p)^\frac1p)\\
\label{eq:hypergeometric}
&=F\left(\frac1p,\frac1p;\frac{2}{p};1-x^p\right),\\
\label{eq:mp}
&=\left(\frac{1+x^p}{2}\right)^{-\frac1p}
F\left(\frac{1}{2p},\frac{1}{2p}+\frac{1}{2};\frac1p+\frac12;
\left(\frac{1-x^p}{1+x^p}\right)^2\right),\\
\frac{1}{\mathrm{K}_p(1,x)}
\label{eq:hypergeometric2}
&=F\left(1,\frac1p;2;1-x^p\right),\\
\label{eq:kp}
&=\left(\frac{1+x^p}{2}\right)^{-\frac1p}
F\left(\frac{1}{2p},\frac{1}{2p}+\frac{1}{2};\frac32;
\left(\frac{1-x^p}{1+x^p}\right)^2\right).
\end{align}
Therefore, Theorem $\ref{thm:BL4}$ immediately follows.
\end{thm}

Moreover, we define a \textit{complete $(p,q)$-elliptic integral of the second kind}
$$\mathcal{E}_{p,q}(k):=\int_0^{\frac{\pi_{p,q}}{2}}
(1-k^q\sin_{p,q}^{q}{\theta})^\frac1p\,d\theta
=\int_0^1 \left(\frac{1-k^qt^q}{1-t^q}\right)^\frac1p\,dt.$$
It is clear that $\mathcal{E}_{2,2}(k)$ is identical to 
the complete elliptic integral of the second kind
$$\mathcal{E}(k)=\int_0^{\frac{\pi}{2}}
\sqrt{1-k^2\sin^2{\theta}}\,d\theta
=\int_0^1 \sqrt{\frac{1-k^2t^2}{1-t^2}}\,dt.$$
Then, we can show the following relation for $p\neq q$.
\begin{thm}
\label{thm:pqL}
Let $p,\ q \in (1,\infty)$ and $k \in [0,1)$. Then
\begin{equation}
\label{eq:pqL}
p\mathcal{E}_{p,q}(k^\frac1q)\mathcal{K}_{q,p}(k^\frac1p)
-q\mathcal{K}_{p,q}(k^\frac1q)\mathcal{E}_{q,p}(k^\frac1p)=\frac{(p-q)\pi_{p,q}\pi_{q,p}}{4}.
\end{equation}
\end{thm}


This paper is organized as follows. In Section 2 we prepare properties
of the complete $(p,q)$-elliptic integrals and show Theorem \ref{thm:pqL}.
Section 3 is devoted to give a proof of Theorem \ref{thm:main}
and an alternative proof of Theorem \ref{thm:BL4}.

Throughout this paper, we write $\mathbb{P}:=(0,1) \cup (1,\infty)$.


\section{Complete $(p,q)$-elliptic integrals}

Let $p$ and $q$ be real numbers satisfying $p^*:=p/(p-1)>0$ and $q>0$
(note that $p$ is allowed to be negative).
The $(p,q)$-trigonometric function $\sin_{p,q}{x}$ is the inverse function of
$$\sin_{p,q}^{-1}{x}:=\int_0^x \frac{dt}{(1-t^q)^\frac1p},\quad x \in [0,1].$$
Clearly, $\sin_{p,q}{x}$ is increasing function from $[0,\pi_{p,q}/2]$ onto $[0,1]$,
where
$$\pi_{p,q}:=2\sin_{p,q}^{-1}{1}
=2 \displaystyle \int_0^1 \dfrac{dt}{(1-t^q)^\frac1p}
=\frac2q B\left(\frac{1}{p^*},\frac1q\right).$$

For $x \in [0,\pi_{p,q}/2)$, we also define 
\begin{align*}
\cos_{p,q}{x}:=(1-\sin_{p,q}^q{x})^{\frac1q},\quad
\tan_{p,q}{x}:=\frac{\sin_{p,q}{x}}{\cos_{p,q}{x}}.
\end{align*}
These functions satisfy, for $x \in (0,\pi_{p,q}/2)$,
\begin{gather*}
\cos_{p,q}^q{x}+\sin_{p,q}^q{x}=1,\label{eq:cs} \\
(\sin_{p,q}{x})'=\cos_{p,q}^{\frac{q}{p}}{x},\notag \\
(\cos_{p,q}{x})'=-\sin_{p,q}^{q-1}{x}\cos_{p,q}^{1-\frac{q}{p^*}}{x}, \notag \\
(\cos_{p,q}^{\frac{q}{p^*}}{x})'=-\frac{q}{p^*}\sin_{p,q}^{q-1}{x}, \notag \\
(\tan_{p,q}{x})'=\cos_{p,q}^{-1-\frac{q}{p^*}}{x}. \notag
\end{gather*}

Now, for any $k \in [0,1)$ 
we define the \textit{complete $(p,q)$-elliptic integral of the first kind}
and \textit{of the second kind}
as follows.
\begin{align*}
\mathcal{K}_{p,q}(k)&:=\int_0^{\frac{\pi_{p,q}}{2}} \frac{d\theta}
{(1-k^q\sin_{p,q}^q{\theta})^{\frac{1}{p^*}}}
=\int_0^1 \frac{dt}{(1-t^q)^\frac1p (1-k^qt^q)^{\frac{1}{p^*}}},\\
\mathcal{E}_{p,q}(k)&:=\int_0^{\frac{\pi_{p,q}}{2}} (1-k^q\sin_{p,q}^q{\theta})^{\frac1p} d\theta
=\int_0^1 \left(\frac{1-k^qt^q}{1-t^q}\right)^\frac1p\,dt.
\end{align*}
It is easy to see that $\mathcal{K}_{p,q}(k)$ is increasing on $[0,1)$ and
$$\mathcal{K}_{p,q}(0)=\frac{\pi_{p,q}}{2},\quad
\lim_{k \to 1-0}\mathcal{K}_{p,q}(k)=\infty,$$
and $\mathcal{E}_{p,q}(k)$ is decreasing on $[0,1)$ and
$$\mathcal{E}_{p,q}(0)=\frac{\pi_{p,q}}{2},\quad
\lim_{k \to 1-0}\mathcal{E}_{p,q}(k)=1,$$

The functions $\mathcal{K}_{p,q}(k)$ and $\mathcal{E}_{p,q}(k)$ satisfy a system of differential equations.
\begin{prop}
\label{prop:p-differential}
$$\frac{d\mathcal{E}_{p,q}}{dk}=\frac{q(\mathcal{E}_{p,q}-\mathcal{K}_{p,q})}{pk},\quad
\frac{d\mathcal{K}_{p,q}}{dk}=\dfrac{\mathcal{E}_{p,q}-(1-k^q)\mathcal{K}_{p,q}}{k(1-k^q)}.$$
\end{prop}

\begin{proof}
Differentiating $\mathcal{E}_{p,q}(k)$ we have
\begin{align*}
\frac{d\mathcal{E}_{p,q}}{dk}
&=\int_0^{\frac{\pi_{p,q}}{2}}
\frac{d}{dk}(1-k^q\sin_{p,q}^q{\theta})^{\frac1p}\,d\theta\\
&=\frac{q}{p}\int_0^{\frac{\pi_{p,q}}{2}}
\dfrac{-k^{q-1}\sin_{p,q}^q{\theta}}{(1-k^q\sin_{p,q}^q{\theta})
^{1-\frac1p}}\,d\theta\\
&=\frac{q}{pk} \left( \int_0^{\frac{\pi_{p,q}}{2}}
\dfrac{1-k^q\sin_{p,q}^q{\theta}}{(1-k^q\sin_{p,q}^q{\theta})^{1-\frac1p}}
\,d\theta
- \int_0^{\frac{\pi_{p,q}}{2}}
\dfrac{d\theta}{(1-k^q\sin_{p,q}^q{\theta})^{1-\frac1p}} \right) \\
&=\frac{q}{pk} (\mathcal{E}_{p,q}-\mathcal{K}_{p,q}).
\end{align*}

Next, for $\mathcal{K}_{p,q}(k)$ 
\begin{align}
\label{eq:dK}
\frac{d\mathcal{K}_{p,q}}{dk}
=\frac{q}{p^*}\int_0^{\frac{\pi_{p,q}}{2}}
\dfrac{k^{q-1}\sin_{p,q}^q{\theta}}
{(1-k^q\sin_{p,q}^q{\theta})^{2-\frac1p}}\,d\theta.
\end{align}
Here we see that 
\begin{align*}
\frac{d}{d\theta} 
&\left(\frac{-\cos_{p,q}^{\frac{q}{p^*}}{\theta}}
{(1-k^q\sin_{p,q}^q{\theta})^{1-\frac1p}}\right)\\
&=\frac{\frac{q}{p^*}\sin_{p,q}^{q-1}{\theta}(1-k^q\sin_{p,q}^q{\theta})
-\frac{q}{p^*}k^q\sin_{p,q}^{q-1}{\theta}\cos_{p,q}^q{\theta}}
{(1-k^q\sin_{p,q}^q{\theta})^{2-\frac1p}}\\
&=\frac{q(1-k^q)\sin_{p,q}^{q-1}{\theta}}
{p^*(1-k^q\sin_{p,q}^q{\theta})^{2-\frac1p}},
\end{align*}
so that we use integration by parts as
\begin{align*}
\frac{d\mathcal{K}_{p,q}}{dk}
&=\int_0^{\frac{\pi_{p,q}}{2}}
\frac{k^{q-1}}{1-k^q}\frac{d}{d\theta} 
\left(\frac{-\cos_{p,q}^{\frac{q}{p^*}}{\theta}}
{(1-k^q\sin_{p,q}^q{\theta})^{1-\frac1p}}\right)
\sin_{p,q}{\theta}\,d\theta\\
&=\frac{k^{q-1}}{1-k^q} \left[\frac{-\cos_{p,q}^{\frac{q}{p^*}}{\theta}
\sin_{p,q}{\theta}}
{(1-k^q\sin_{p,q}^q{\theta})^{1-\frac1p}}
\right]_0^{\frac{\pi_{p,q}}{2}}
+\frac{k^{q-1}}{1-k^q} \int_0^{\frac{\pi_{p,q}}{2}} 
\frac{\cos_{p,q}^q{\theta}}{(1-k^q\sin_{p,q}^q{\theta})^{1-\frac1p}}
\,d\theta\\
&=\frac{k^{q-1}}{1-k^q} \int_0^{\frac{\pi_{p,q}}{2}}
\frac{1}{k^q} \cdot
\frac{1-k^q\sin_{p,q}^q{\theta}-(1-k^q)}{(1-k^q\sin_{p,q}^q{\theta})^{1-\frac1p}}\,d\theta\\
&=\frac{1}{k(1-k^q)}(\mathcal{E}_{p,q}-(1-k^q)\mathcal{K}_{p,q}).
\end{align*}
This completes the proof.
\end{proof}


Proposition \ref{prop:p-differential} now yields 
Theorem \ref{thm:pqL}.

\begin{proof}[Proof of Theorem \ref{thm:pqL}]
We will differentiate the left-hand side of \eqref{eq:pqL}
and apply Proposition \ref{prop:p-differential}.
A direct computation shows that
\begin{align*}
\frac{d}{dk}(p\mathcal{E}_{p,q}(k^\frac1q) & \mathcal{K}_{q,p}(k^\frac1p)
-q\mathcal{K}_{p,q}(k^\frac1q)\mathcal{E}_{q,p}(k^\frac1p))\\
& =p \cdot \frac{1}{pk}(\mathcal{E}_{p,q}(k^\frac1q)-\mathcal{K}_{p,q}(k^\frac1q))\cdot \mathcal{K}_{q,p}(k^\frac1p)\\
& \quad \quad  
+p\mathcal{E}_{p,q}(k^\frac1q)\cdot \frac{1}{pk(1-k)}(\mathcal{E}_{q,p}(k^\frac1p)-(1-k)\mathcal{K}_{q,p}(k^\frac1p))\\
& \quad \quad \quad -q \cdot \frac{1}{qk(1-k)}(\mathcal{E}_{p,q}(k^\frac1q)-(1-k)\mathcal{K}_{p,q}(k^\frac1q))
\cdot \mathcal{E}_{q,p}(k^\frac1p)\\
& \quad \quad \quad \quad \quad -q\mathcal{K}_{p,q}(k^\frac1q)\cdot \frac{1}{qk}(\mathcal{E}_{q,p}(k^\frac1p)
-\mathcal{K}_{q,p}(k^\frac1p))\\
&=0.
\end{align*}
Therefore the left-hand side of \eqref{eq:pqL} is a constant $C$.
Letting $k=0$, we obtain
$$C=p\frac{\pi_{p,q}}{2}\frac{\pi_{q,p}}{2}-q\frac{\pi_{p,q}}{2}\frac{\pi_{q,p}}{2}
=\frac{(p-q)\pi_{p,q}\pi_{q,p}}{4},$$ 
and the proof is complete.
\end{proof}


For a real number $a$ and a natural number $n$,
we define
$$(a)_n:=\frac{\Gamma(a+n)}{\Gamma(a)}
=(a+n-1)(a+n-2)\cdots (a+1) a.$$
We adopt the convention that $(a)_0:=1$.
For $|x|<1$ the series
$$F(a,b;c;x):=\sum_{n=0}^\infty \frac{(a)_n(b)_n}{(c)_n}\frac{x^n}{n!}$$
is called a \textit{Gaussian hypergeometric series}. 
See \cite{AAR} for more details.

\begin{lem}
\label{lem:integral}
For $n=0,1,2,\ldots$
$$\int_0^{\frac{\pi_{p,q}}{2}} \sin_{p,q}^{qn}{\theta}\,d\theta
=\frac{\pi_{p,q}}{2}\frac{(\frac1q)_n}{(\frac{1}{p^*}+\frac1q)_n}.$$
\end{lem}

\begin{proof}
Letting $\sin_{p,q}^q{\theta}=t$, we have 
$$
\int_0^{\frac{\pi_{p,q}}{2}} \sin_{p,q}^{qn}{\theta}\,d\theta
=\frac1q \int_0^1 t^{n+\frac1q-1}(1-t)^{-\frac1p}\,dt
=\frac1q B\left(n+\frac1q,\frac{1}{p^*} \right).
$$
Moreover,
\begin{align*}
\frac1q B\left(n+\frac1q,\frac{1}{p^*} \right)
&=\frac1q B\left(\frac1q,\frac{1}{p^*} \right)
\frac{B\left( n+\frac1q,\frac{1}{p^*} \right)}
{B\left(\frac1q,\frac{1}{p^*} \right)}\\
&=\frac{\pi_{p,q}}{2}\frac{\Gamma(n+\frac1q)\Gamma(\frac1q+\frac{1}{p^*})}
{\Gamma(\frac1q)\Gamma(n+\frac1q+\frac{1}{p^*}) }\\
&=\frac{\pi_{p,q}}{2}\frac{(\frac1q)_n}{(\frac{1}{p^*}+\frac1q)_n},
\end{align*}
and the lemma follows.
\end{proof}

\begin{prop}
\label{prop:hypergeometric}
\begin{align*}
\mathcal{K}_{p,q}(k)
&=\frac{\pi_{p,q}}{2} F\left(\frac{1}{p^*},\frac1q;\frac{1}{p^*}+\frac1q;k^q\right),\\
\mathcal{E}_{p,q}(k)
&=\frac{\pi_{p,q}}{2} F\left(-\frac1p,\frac1q;\frac{1}{p^*}+\frac1q;k^q\right).
\end{align*}

\end{prop}

\begin{proof}
Binomial series expansion gives
\begin{align*}
\mathcal{K}_{p,q}(k)
=\int_0^{\frac{\pi_{p,q}}{2}} 
(1-k^q\sin_{p,q}^q{\theta})^{-\frac{1}{p^*}}\,d\theta
=\sum_{n=0}^\infty (-1)^n \binom{-\frac{1}{p^*}}{n} k^{qn}
\int_0^{\frac{\pi_{p,q}}{2}} \sin_{p,q}^{qn}{\theta}\,d\theta.
\end{align*}
Here, using Lemma \ref{lem:integral} and the fact
\begin{align*}
(-1)^n \binom{-\frac{1}{p^*}}{n}
=\frac{(\frac{1}{p^*})_n}{n!},
\end{align*}
we see that 
$$\mathcal{K}_{p,q}(k)=\frac{\pi_{p,q}}{2}\sum_{n=0}^\infty
\frac{(\frac{1}{p^*})_n (\frac1q)_n}{(\frac{1}{p^*}+\frac1q)_n}\frac{k^{qn}}{n!}
=\frac{\pi_{p,q}}{2} F\left(\frac{1}{p^*},\frac1q;\frac{1}{p^*}+\frac1q;k^q
\right).$$
The proof of $\mathcal{E}_{p,q}(k)$ is similar, so that we omit it.
\end{proof}

\section{Proof of Theorem \ref{thm:main}}

By properties (ii) and (iii) of binary symmetric mean in the introduction, 
we may assume that $a \geq b>0$ and it is enough to consider
$\mathrm{M}_p(1,x)$ for any $x \in (0,1]$ instead of $\mathrm{M}_p(a,b)$.

The following is a fundamental quadratic transformation of hypergeometric functions.
For the proof, see for instance \cite[Theorem 3.1.3]{AAR}.

\begin{lem}
\label{lem:aar}
For all $x$ where the series converge
$$F(a,b;2a;x)=\left(1-\frac{x}{2}\right)^{-b}
F\left(\frac{b}{2},\frac{b+1}{2};a+\frac12;\left(\frac{x}{2-x}\right)^2\right).$$
\end{lem}

Now we are in a position to prove Theorem \ref{thm:main}.
Let $p \in \mathbb{P}$.
Setting $t=x\tan_{p^*,p}{\theta}$ in the right-hand side of
$$\frac{1}{\mathrm{M}_p(1,x)}=c_p \int_0^\infty
\frac{dt}{((t^p+1)(t^p+x^p))^{\frac{1}{p}}},$$
we have
\begin{align*}
\frac{1}{\mathrm{M}_p(1,x)}
&=c_p
\int_0^{\frac{\pi_{p^*,p}}{2}}
\frac{x\cos_{p^*,p}^{-2}{\theta}\,d\theta}{(x^p\tan_{p^*,p}^p{\theta}+1)^\frac1p
(x^p\tan_{p^*,p}^p{\theta}+x^p)^\frac1p}\\
&=c_p
\int_0^{\frac{\pi_{p^*,p}}{2}}
\frac{d\theta}{(\cos_{p^*,p}^p{\theta}+x^p\sin_{p^*,p}^p{\theta})^{\frac1p}}\\
&=c_p
\int_0^{\frac{\pi_{p^*,p}}{2}}
\frac{d\theta}{\left(1-(1-x^p)\sin_{p^*,p}^p{\theta}\right)^{\frac1p}}\\
&=c_p
\mathcal{K}_{p^*,p}((1-x^p)^{\frac1p}),
\end{align*} 
where 
$$\frac{1}{c_p}=\int_0^\infty \frac{dt}{(1+t^p)^{\frac{2}{p}}}
=\frac{\pi_{p^*,p}}{2}.$$
Thus, Proposition \ref{prop:hypergeometric} yields \eqref{eq:hypergeometric},
i.e.,
$$\frac{1}{\mathrm{M}_p(1,x)}
=F\left(\frac1p,\frac1p;\frac{2}{p};1-x^p\right).$$
Applying Lemma \ref{lem:aar}
with $a=b=1/p$ and $x$ replaced by $1-x^p$ to \eqref{eq:hypergeometric},
we have \eqref{eq:mp}.

Next, recall that $1/\mathrm{K}_p(1,x)$ can be written as \eqref{eq:BLkp}.
As in the proof of Proposition \ref{prop:hypergeometric},
we obtain
\begin{align*}
\frac{1}{\mathrm{K}_p(1,x)}
&=\int_0^1 (1-(1-x^p)s)^{-\frac1p}\,ds\\
&=\sum_{n=0}^\infty (-1)^n \binom{-\frac1p}{n} (1-x^p)^n \int_0^1 s^n\,ds\\
&=\sum_{n=0}^\infty \frac{(\frac1p)_n}{(n+1)!} (1-x^p)^n\\
&=F\left(1,\frac1p;2;1-x^p\right),
\end{align*}
which implies \eqref{eq:hypergeometric2}.
Applying Lemma \ref{lem:aar} with $a=1,\ b=1/p$ and $x$ replaced
by $1-x^p$ to the last series, we have \eqref{eq:kp}.
Therefore, we accomplished the proof of Theorem \ref{thm:main}.
\medskip

Theorem \ref{thm:BL4} immediately follows from Theorem \ref{thm:main}. 
Indeed, we assume $p \in \mathbb{P}$.
Comparing the third parameters of \eqref{eq:mp} and \eqref{eq:kp}, we can see that 
$$\frac{1}{\mathrm{M}_p(1,x)} \gtrless \frac{1}{\mathrm{K}_p(1,x)} \quad
\Leftrightarrow \quad
\frac1p+\frac12 \lessgtr \frac32,$$
hence
$$\mathrm{M}_p(1,x) \gtrless \mathrm{K}_p(1,x) \quad
\Leftrightarrow \quad
p \lessgtr 1.$$
We leave it to the reader to verify that $\mathrm{M}_0(a,b)>\mathrm{K}_0(a,b)$. 


\begin{rem}
Motivated by the expression in \cite{BL}:
\begin{equation}
\label{eq:BL1}
\frac{1}{\mathrm{M}_p(a,b)}=(\max\{a,b\})^{-1}
\sum_{k=0}^\infty \prod_{i=0}^{k-1}
\frac{(\frac1p+i)^2}{\frac2p+i} \frac{1}{k!}
\left[1-\left(\frac{\min\{a,b\}}{\max\{a,b\}}\right)^p\right]^k,
\end{equation}
Nakamura \cite[Remark 3.9]{Na} also indicates
that $1/\mathrm{M}_p(1,x)$ is nothing but \eqref{eq:hypergeometric}.
On the other hand, our proof gives \eqref{eq:hypergeometric}
without deducing \eqref{eq:BL1}.
\end{rem}

\begin{rem}
Each mean of $\mathrm{L}(a,b),\ \mathrm{AG}(a,b)$ and $\mathrm{K}_p(a,b)$
has the other characterization than the integral expression. 
It is of interest to characterize Bhatia-Li's mean $\mathrm{M}_p(a,b)$ 
with no use of integral,
but we have not been able to do this. 
\end{rem}

%
%
%
%
%
%
%
%
%








\begin{thebibliography}{99}


\bibitem{AAR}
G.\,Andrews, R.\,Askey and R.\,Roy,
\textit{Special functions},
Encyclopedia of Mathematics and its Applications, 71. 
Cambridge University Press, Cambridge, 1999.




 

\bibitem{BL}
R.\,Bhatia and R.-C.\,Li,
An interpolating family of means,
\textit{Commun.\,Stoch.\,Anal.} \textbf{6} (2012), no. 1, 15--31. 






\bibitem{BB3} J.M.\,Borwein and P.B.\,Borwein,
\textit{Pi and the AGM}, 
A study in analytic number theory and computational complexity. 
Reprint of the 1987 original. Canadian Mathematical Society Series 
of Monographs and Advanced Texts, 4. A Wiley-Interscience Publication. 
John Wiley \& Sons, Inc., New York, 1998

%

%
%








\bibitem{DM} P.\,Dr\'{a}bek and R.\,Man\'{a}sevich,
On the closed solution to some nonhomogeneous eigenvalue problems 
with $p$-Laplacian,
\textit{Differential Integral Equations}
\textbf{12} (1999), 773--788. 

%


%
%







\bibitem{LE} J.\,Lang and D.E.\,Edmunds,
\textit{Eigenvalues, embeddings and generalised trigonometric functions}, 
Lecture Notes in Mathematics, 2016. Springer, Heidelberg, 2011. 


%



\bibitem{Na} N.\,Nakamura,
Order relations among some interpolating families of means,
\textit{Toyama Math.\,J.} \textbf{35} (2012), 35--48.
 






\bibitem{T} S.\,Takeuchi,
Generalized Jacobian elliptic functions and their application to bifurcation problems 
associated with $p$-Laplacian, 
\textit{J.\,Math.\,Anal.\,Appl.}\,\textbf{385} (2012), no.\,1, 24--35. 

\bibitem{T2} S.\,Takeuchi,
The basis property of generalized Jacobian elliptic functions,
\textit{Commun.\,Pure Appl.\,Anal.} \textbf{13} (2014), no.\,6, 2675--2692.

\bibitem{T3} S.\,Takeuchi,
A new form of the generalized complete elliptic integrals,
to appear in \textit{Kodai Math.\ J.}



%
%
%

 

\end{thebibliography}
\end{document}